\documentclass[11pt,reqno]{amsart}

\usepackage[comma, sort&compress]{natbib}

\usepackage[a4paper,vmargin={3.5cm,3.5cm},hmargin={2.5cm,2.5cm},bottom=30mm]{geometry}

\usepackage{booktabs} 
\usepackage[final,commentmarkup=footnote,authormarkup=none]
{changes}

\definechangesauthor[name={Ale}, color=red]{ale}
\definechangesauthor[name={Ale}, color=blue]{ale2}
\definechangesauthor[name={Ale}, color=orange]{ale3} 
\definechangesauthor[name={Wio}, color=magenta]{wio}

\usepackage{array} 
\usepackage{paralist} 
\usepackage{verbatim} 
\usepackage{subfig} 
\usepackage[T1]{fontenc} 

\def\1{{\mathchoice {\rm 1\mskip-4mu l} {\rm 1\mskip-4mu l}
{\rm 1\mskip-4.5mu l} {\rm 1\mskip-5mu l}}}

\usepackage{pxfonts}
\usepackage{dsfont}

\usepackage{hyperref}
\hypersetup{ colorlinks=true, linkcolor=blue, citecolor=blue,
  filecolor=blue, urlcolor=blue}
\usepackage{tikz}

\usepackage[comma,sort&compress]{natbib}

\newcommand{\RR}{\mathbb{R}}
\newcommand{\eps}{\varepsilon}
\newcommand{\NN}{\mathbb{N}}

\newcommand{\PP}{\mathbb{P}}

\newcommand{\TT}{\mathbb{T}}
\newcommand{\ZZ}{\mathbb{Z}}

\newcommand{\EE}{\mathbb{E}}

\newcommand{\Bb}{\mathcal{B}}
\newcommand{\ee}{\mathrm{e}}

\newcommand{\De}{\mathrm{d}}

\renewcommand{\i}{{\imath}}

\newtheorem{theorem}{Theorem}
\newtheorem{proposition}[theorem]{Proposition}

\newtheorem{lemma}[theorem]{Lemma}
\newtheorem{remark}[theorem]{Remark}

\theoremstyle{definition}
\newtheorem{definition}{Definition}[section]

\usepackage{amssymb}
\numberwithin{equation}{section}

\title[Scaling in divisible sandpiles: a Fourier multiplier approach]{Scaling limits in divisible sandpiles: a Fourier multiplier approach}
\author[A. Cipriani]{Alessandra Cipriani}
\author[J. de Graaff]{Jan de Graaff}
\author[W.~M.~Ruszel]{Wioletta M.~Ruszel}
\address{TU Delft (DIAM), Building 28, van Mourik Broekmanweg 6, 2628 XE, Delft, The Netherlands}
\email{A.Cipriani@tudelft.nl, Jan\_de\_Graaff@hotmail.com, W.M.Ruszel@tudelft.nl}
\thanks{The first author acknowledges the support of the grant 613.009.102 of the Netherlands Organisation for
Scientific Research (NWO). The first and third author would like to thank Leandro Chiarini and Rajat Subhra Hazra for helpful discussions.}

\begin{document}
\begin{abstract}
In this paper we investigate scaling limits of the odometer in divisible sandpiles on $d$-dimensional tori \replaced[id=ale3,remark={nr1}]{following up}{generalising} the works \cite{CHR17,LongRange,CHRheavy}. Relaxing the assumption of independence of the weights of the divisible sandpile, we generate generalised Gaussian fields in the limit by specifying the Fourier multiplier of their covariance kernel. In particular, using a Fourier multiplier approach, we can recover fractional Gaussian fields of the form $(-\Delta)^{-s/2} W$ for $s>2$ and $W$ a spatial white noise on the $d$-dimensional unit torus.
\end{abstract}
\maketitle

\section{Introduction and main results}

Gaussian random fields arise naturally in the study of many statistical physical models. In particular fractional Gaussian fields $FGF_s(D):=(-\Delta)^{-s/2} W$, where $W$ denotes a spatial white noise, $s\in \RR$ and $D\subset \mathbb{R}^d$, typically arise in the context of random phenomena with long-range dependence and are closely related to renormalization. Examples of fractional Gaussian fields include the Gaussian free field and the continuum bi-Laplacian model. We refer the reader to~\cite{LSSW} and references therein for a complete survey on fractional Gaussian fields.
In this paper we study a class of divisible sandpile models and show that the scaling limit of its odometer functions converges to a Gaussian limiting field indexed by a Fourier multiplier. 

The divisible sandpile was introduced by~\cite{LevPer,LePe10} and it is defined as follows. A divisible sandpile configuration on the discrete torus $\ZZ_n^d$ of side-length $n$ is a function $s : \ZZ_n^d \to\RR$, where $s(x)$ indicates a mass of particles or a hole at site $x$. 
Note that here, unlike the classical Abelian sandpile model~\citep{BTW,JaraiCornell}, $s(x)$ is a real-valued number. 
Given $(\sigma(x))_{x\in \ZZ_n^d}$ a sequence of centered (possibly correlated) multivariate Gaussian random variables, we choose $s$ to be equal to 
\begin{equation}\label{eq:def_s}
s(x) = 1 + \sigma(x) - \frac{1}{n^d} \sum_{z\in \mathbb{Z}^d_n} \sigma(z).
\end{equation}
If a vertex $x \in \ZZ_n^d$ is unstable, i.e. $s(x)> 1$, it topples by keeping mass $1$ for itself and distributing the excess $s(x) -1$ uniformly among its neighbours. At each discrete time step, all unstable vertices topple simultaneously. The configuration $s$ defined as \eqref{eq:def_s} will stabilize to the all 1 configuration. The odometer $u_n:\mathbb{Z}^d_n\rightarrow \mathbb{R}_{\ge 0}$ collects the information about all mass which was emitted from each vertex in $\mathbb{Z}^d_n$ during stabilization.
Our main theorem states that $u_n$, properly rescaled, converges to a Gaussian random field in some appropriate Sobolev space.

In \cite{CHR17} the authors consider divisible sandpiles with nearest-neighbor mass distribution, and show that for any configuration $s$ given by \eqref{eq:def_s} where the $\sigma$'s are i.i.d. with finite variance the limiting odometer is a bi-Laplacian Gaussian field $(-\Delta)^{-1} W$ on the unit torus $\TT^d$ (or $FGF_2(\TT^d)$ in the notation of~\cite{LSSW}). Relaxing the second moment assumption on $\sigma$ leads to limiting fields which are no longer Gaussian, but alpha-stable random fields, see \cite{CHRheavy}. On the other side if one keeps the second moment assumption and instead redistributes the mass upon toppling to all neighbours following the jump distribution of a long-range random walk one can construct fractional Gaussian fields with $0< s\leq 2$ (\cite{LongRange}). To summarize, Gaussian fields appear under the assumption of finite second moments in the initial configuration, while tuning the redistribution of the mass leads to limiting interfaces with smoothness which is at most the one of a bi-Laplacian field. 

One natural question arises: what kind of sandpile models give rise to odometer interfaces which are smoother that the bi-Laplacian? It turns out that to obtain limiting fields of the form $(-\Delta)^{-s/2} W$ such that $s>2$ the long-range dependence must show up in the initial Gaussian multivariate variables $\sigma$ rather than in the redistribution rule.
The novelty of the present article is that it complements \cite{LongRange,CHR17, CHRheavy} by removing the assumption of independence of the weights (in the Gaussian case), and in addition provides an example of a model defined on a discrete space which scales to a limiting field $(-\Delta)^{-s/2} W$ such that $s>2$. Note \added[id=ale2]{that for the fractional case with $s>2$ one does not have the aid of explicit integral representations for the eigenvalues, and therefore constructing the continuum field via a discrete approximation requires new approaches. To the best of the authors' knowledge this article is the first instance of such a construction.}

{In the proof, we start by defining a sequence of covariance matrices $(K_n)_{n\in\NN}$ for the weights of the sandpile. Their Fourier transform $\widehat{K_n}$ is assumed to have a pointwise limit $\widehat K$ as $n$ goes to infinity. Under suitable regularity assumptions, $\widehat K$ defines the Fourier multiplier of the covariance kernel of the limiting field. A key idea in the proof is to absorb the multiplier $\widehat K$ into the definition of the abstract Wiener space.} {This defines a new Hilbert space where we will construct the limiting field.}
Note that this approach is different from the one in \cite{CHR17, LongRange}, where the covariance structure of the odometer was given. Furthermore we would like to stress that the scaling factor $a_n\sim n^{-2}$ used for convergence (see Theorem~\ref{main}) is dimension-independent in contrast to the above mentioned works. This follows from the fact that, being the $\sigma$'s correlated, the dimensional scaling is absorbed in the covariance structure of the odometer (see Lemma~\ref{covv}).

Let us finally remark that depending on the parameters $s,\,d$ the limiting field will be either a random distribution or a random continuous function. More precisely, if the Hurst parameter $H$ of the $FGF_s$  field
\[
 H:=s-\frac{d}{2}
\]
is strictly negative, then the limit is a random distribution while for $H\in (k,\,k+1)$, $k\in \NN\cup\{0\}$, the field is a $(k-1)$-differentiable function (\cite{LSSW}, with the caveat that the results presented therein are worked out for $\RR^d$ or domains with zero boundary conditions).
In the case of $H\ge 0$, a stronger result could be pursued, namely an invariance principle \`a-la Donsker (as for example in~\citet[Theorem~2.1]{CDHMem},~\citet[Theorem~3]{CDHMix}).  To keep the same outline for all proofs we will treat the limiting field a priori as a random distribution and thus prove finite dimensional distribution convergence by testing the rescaled odometer against suitable test functions.

\subsection{Main result}
\paragraph{\em Notation}
In all that follows, we will consider $d\ge 1$. We are going to work with the spaces $\ZZ_n^d:=[-n/2,\,n/2]^d\cap\ZZ^d$, the discrete torus of side-length $n$, and $\TT^d:=[-1/2,\,1/2]^d$, the $d$-dimensional  torus. Moreover let $B(z,\,\rho)$ a ball centered at $z$ of radius $\rho>0$ in the $\ell^\infty$-metric. We will use throughout the notation $z\cdot w$ for the Euclidean scalar product between $z,\,w\in \RR^d$. We will let $C,\,C',\,\,c\ldots$ be positive constants which may change from line to line within the same equation. We define the Fourier transform of a function $f\in L^1(\TT^d)$ as $\widehat f(y):=\int_{\TT^d}f(z)\exp\left(-2\pi\i y\cdot z\right)\De z$ for $y\in \ZZ^d$. We will use the symbol $\,\widehat \cdot\,$ to denote also Fourier transforms on $\ZZ_n^d$ and $\RR^d$ (cf. Subsection~\ref{subsec-fourier-torus} for the precise definitions). 
\vspace{\baselineskip}

We can now state our main result. We consider the piecewise interpolation of the odometer on small boxes of radius $1/ 2n$ and show convergence to the limiting Gaussian field $\Xi^K$ depending on the covariance $K_n$ of the initial sandpile configuration. This field can be represented in several ways: a convenient one is to let it be the Gaussian field with characteristic functional
\begin{equation}\label{eq:def_Xi}
	\Phi(f):=\exp\left(-\frac{\|f\|_{K}^2}{2}\right)
\end{equation}
where $f$ belongs to the Sobolev space $H_K^{-1}(\TT^d)$ with norm
\[
\|f\|^{\added[id=wio,remark=]{2}}_{K}:=\sum_{\xi\in\ZZ^d\setminus\{0\}} \widehat{K}(\xi) \|\xi\|^{-4} |\widehat{f}(\xi) |^2.
\]
We will give the analytical background to this definition in Subsection~\ref{subsec:AWS}. Note that we index the norm and the Sobolev space by the Fourier multiplier $K$.
\begin{theorem}\label{main}
For $n\in \mathbb{N}$, consider an initial sandpile configuration defined by \eqref{eq:def_s} where the collection of centered Gaussians $(\sigma(x))_{x\in\ZZ_n^d}$ has covariance
\[ \EE[\sigma(x) \sigma(y)] = K_n(x-y). \]
Assume that $\widehat{K_n}$, the Fourier transform of $K_n$ on $\ZZ_n^d$, satisfies
\[
 \sup_{n\in\NN}\sup_{\xi\in\ZZ^d}\widehat{K_n}(\xi)<\infty, 
\]
and that 
\begin{equation}\label{eq:K_hat}\lim_{n\rightarrow \infty} \widehat{K_n}(\xi)=:\widehat K(\xi) >0,\quad\xi\in\ZZ^d\end{equation} 
exists.
Let $(u_n(x))_{x\in\ZZ_n^d}$ be the odometer associated with the collection $(\sigma(x))_{x\in\ZZ^d_n}$ via~\eqref{eq:def_s}. 
Let furthermore $a_n:= 4\pi^2(2d)^{-1} n^{-2}$. We define the formal field on $\TT^d$ by
\[ \Xi^K_n(x):=  \sum_{z\in\TT_n^d} u_n(nz) \1_{B(z,\frac{1}{2n})}(x), \ \ \ x\in\TT^d. \]
Then $a_n \Xi^K_n(x) $ converges in law as $n\to\infty$ to $\Xi^K$ in the topology of the Sobolev space $H_K^{-\eps}(\TT^d)$, where 
\[
 \eps>\max\left\{\frac{d}{2},\,\frac{d}{4}+1\right\}
\]
(for the analytical specification see Subsection~\ref{subsec:AWS}). \added[id=wio,remark=]{The field $\Xi^K$ and the space $H^{-\eps}_K(\TT^d)$ depend on $K$, the inverse Fourier transform of $\widehat{K}$ as in \eqref{eq:K_hat}.}
\end{theorem}

\paragraph{\em Structure of the paper} We will give an overview of the needed results on divisible sandpiles and Fourier analysis on the torus in Section~\ref{sec:notation}. The proof of the main result will be shown in Section~\ref{sec:proof_main}. In Section~\ref{sec:examples} we discuss two classes of examples. In the first class, we consider weights with summable covariances, leading to a bi-Laplacian scaling limit. In the second class the limiting odometer is a fractional field of the form $(-\Delta)^{- s/2 }W$, $s>2$.
\section{Preliminaries}\label{sec:notation}
\subsection{Fourier analysis on the torus}
\label{subsec-fourier-torus}
We will use the following inner product for  $\ell^2(\mathbb{Z}^d_n)$:

\[
	\langle {f,g} \rangle 
	=
	\frac{1}{n^d} \sum_{z \in \mathbb{Z}^d_n} 
	f(z) \overline{g(z)}.
\]
\added[id=wio,remark=]{Let $\Delta_g$ denote the graph Laplacian defined by
\[
\Delta_g f(x) = \frac{1}{2d} \sum_{\|y - x\|=1} f(y) - f(x).
\]
}
Consider the Fourier basis of the same space  given by the \added[id=ale3,remark={nr4. It is right that eigenfunctions have a plus sign}]{eigenfunctions of the Laplacian} $\{\psi_w\}_{w\in \mathbb{Z}^d_n}$ with
\begin{equation}\label{def-fourier-basis-discrete}
	\psi_w(z)=\psi^{(n)}_w(z):=\exp \bigg( 2\pi \i z\cdot\frac{w}{n}\bigg).
\end{equation}
The corresponding eigenvalues $\{\lambda_w\}_{w\in \mathbb{Z}^d_n}$ are given by
\begin{equation}\label{eq:eigenvalues}
 \lambda_w:=- \frac{4}{2d}\sum_{i=1}^d\sin^2\bigg(\frac{\pi w_i}{n}\bigg).
\end{equation}
Given $f \in \ell^2(\mathbb{Z}^d_n)$, we define its discrete Fourier transform by
\[
	\widehat{f}(w) 
=
	\langle  f, \psi_{w}\rangle
= 
	\frac{1}{n^d}\sum_{z \in \mathbb{Z}^d_n} 
	f(z) \exp \bigg(-2\pi \i z\cdot\frac{w}{n}\bigg)  
\]
for $w\in \mathbb{Z}^d_n$.
Similarly, if $f, g \in L^2(\mathbb{T}^d)$ we will denote
\[
	(f,g)_{L^2(\mathbb{T}^d)}:= \int_{\mathbb{T}^d} f(z) \overline{g(z)}
	\text{d}z.
\]
Consider the Fourier basis $\{\phi_{\xi}\}_{{\xi} \in \mathbb{Z}^d}$ of $L^2(\TT^d)$ given by 
$$\phi_{\xi}(x):=\exp(\added[id=ale3,remark={nr4}]{}2\pi \i \xi \cdot x)$$ 
and denote 
\[
	\widehat{f}(\added[id=wio,remark=]{\xi}):=(f,\phi_{\added[id=wio,remark=]{\xi}})_{L^2(\mathbb{T}^d)} = 
	\int_{\mathbb{T}^d} f(z) \ee^{- 2 \pi \i \added[id=wio,remark=]{\xi} \cdot z} \text{d}z.
\]
It is important to notice that for $f \in C^\infty(\mathbb{T}^d)$, if we define $f_n: \mathbb{Z}^d_n \to \mathbb{R}$ by $f_n(z) := f( {z}/{n})$, then for all $\added[id=wio,remark=]{\xi} \in \mathbb{Z}^d$, $\widehat{f_n}(\added[id=wio,remark=]{\xi}) \to \widehat{f}(\added[id=wio,remark=]{\xi})$
as $n\rightarrow \infty$. 

Finally, we write $C^\infty(\TT^d)/\sim$ for the space of smooth functions with zero mean, that is, the space of smooth functions modulo the equivalence relation of differing by a constant. 

\subsection{Abstract Wiener Spaces and continuum fractional Laplacians}\label{subsec:AWS}
In this Subsection our aim is to define the appropriate negative Sobolev space in which the convergence of Theorem~\ref{main} occurs. To do so, we repeat the classical construction of abstract Wiener spaces as done in \cite{Silvestri, CHR17}.

\begin{lemma}
Let $\widehat{K}:\ZZ^d\to\RR_{>0}$ be the limiting Fourier multiplier \added[id=ale3,remark={nr5}]{as defined in Theorem~\ref{main}}.
For $a < 0$ the following
\begin{equation}\label{mad_a}
(f,g)_{K,a} := \sum_{\xi\in\ZZ^d\setminus\{0\}} \widehat{K}(\xi) \|\xi\|^{4a} \widehat{f}(\xi) \overline{\widehat{g}(\xi)}
\end{equation}
is an inner product on $C^{\infty}(\TT^d)/\sim.$
\end{lemma}
\begin{proof}
The linearity and conjugate symmetry are immediate. Furthermore, since \eqref{eq:K_hat} holds it follows that 
\[ (f,f)_{K,a} = \sum_{\xi\in\ZZ^d\setminus\{0\}} \widehat{K}(\xi) \|\xi\|^{4a} |\widehat{f}(\xi)|^2 > 0, \]
where the sum converges because $a<0$ and $f\in (C^\infty(\TT^d)/\sim) \subset (L^2(\TT^d)/\sim)$. On the other hand, if we have $(f,f)_{K,a}=0$ then we must have $\widehat{f}(\xi)=0$ for all $\xi\in\ZZ^d\setminus\{0\}$ and so $f \equiv 0$. 
\end{proof}

Define $H^a_K(\TT^d)$ to be the Hilbert space completion of $C^\infty(\TT^d)/\sim$ with respect to the norm $\|\cdot\|_{K,a}$. Our goal is to define a Gaussian random variable $\Xi^K$ such that for all $f\in C^{\infty}(\TT^d)/\sim$ we have $\langle\Xi^K,f\rangle\sim\mathcal{N}(0,\|f\|_{K,a}^2)$. We do this by constructing an appropriate abstract Wiener space for $\Xi^K$. We first of all recall the definition of such a space (see~\citet[\S 8.2]{Str08}). 

\begin{definition}\label{aws2}
A triple $(H,B,\mu)$ is called an abstract Wiener space (from now on abbreviated {AWS}) if
\begin{enumerate}
\item $H$ is a Hilbert space with inner product $(\cdot,\cdot)_H$.
\item $B$ is the Banach space completion of $H$ with respect to the measurable norm $\|\cdot\|_B$. Furthermore $B$ is supplied with the Borel $\sigma$-algebra $\Bb$ induced by $\|\cdot\|_B$.
\item $\mu$ is the unique probability measure on $B$ such that for all $\phi\in B^*$ we have $\mu\cdot\phi^{-1} = \mathcal{N}(0,\|\widetilde{\phi}\|_H^2)$, where $\widetilde{\phi}$ is the unique element of $H$ such that $\phi(h)=(\widetilde{\phi},h)_H$ for all $h\in H$.
\end{enumerate}
\end{definition}

\noindent In order to construct a measurable norm $\|\cdot\|_B$ as above it is sufficient to construct a Hilbert--Schmidt operator on $H$ and set $\|\cdot\|_B := \|T\cdot\|_H$. 
For $a\in \mathbb{R}$ define the \textit{continuum fractional Laplace operator} $(-\Delta)^a$ acting on $L^2(\mathbb{T}^d)$ functions $f$ w.r.t. the orthonormal basis $\{\phi_{\nu}\}_{\nu \in \mathbb{Z}^d}$ as
\[
(-\Delta)^a f(x) := \sum_{\nu\in \mathbb{Z}^d\setminus \{0\}} \|\nu\|^{2a} \widehat{f}(\nu) \phi_{\nu}(x).
\] 
We would like to make two remarks at this point. The first one is that a priori the above operator is not defined for all $L^2(\mathbb{T}^d)$-functions for all values of $a\in \mathbb{R}$. In fact we will construct appropriate Sobolev spaces formally consisting of $L^2(\mathbb{T}^d)$ functions $f$ such that $(-\Delta)^a f(x)$ is again square-integrable.
The second remark concerns the need mean zero test functions in order to cancel the atom at $\nu=0$ which arises from taking an inverse ($a<0$) of the Laplacian in the definition above.

In the following \replaced[id=ale2]{Lemma we will construct an orthonormal basis on $H^{a}_K(\TT^d)$}{ we construct an orthonormal basis on $H^{a}_K(\TT^d)$} \deleted[id=wio,remark=]{and prove that it is indeed a Hilbert space}. We set 
\begin{align*}
f_\xi &:=\widehat{K}(\xi)^{-1/2} (-\Delta)^{-a} \phi_\xi = \widehat{K}(\xi)^{-1/2} \|\xi\|^{-2a} \phi_\xi. 
\end{align*}

\begin{lemma}
$\{ f_\xi \}_{\xi\in\ZZ^d\setminus\{0\}}$ is an orthonormal basis of $H^{a}_K(\TT^d)$ under the norm $\|\cdot\|_{K,a}$.
\end{lemma}
\begin{proof}
First we observe that the $f_\xi$'s are orthogonal:
\[ 
\begin{split}
(f_k,f_\ell)_{K,a} &= \sum_{\xi\in\ZZ^d\setminus\{0\}} \widehat{K}(\xi) \|\xi\|^{4a} \left(\widehat{K}(\xi)\right)^{-1/2} \|\xi\|^{-2a} \1_{\xi=k}\left(\widehat{K}(\xi)\right)^{-1/2} \|\xi\|^{-2a} \1_{\xi=\ell}  = \1_{k=\ell}. 
\end{split}
\]
Next we show that all $g\in H^a_K(\TT^d)$ have a Fourier expansion in the $f_\xi$'s. Indeed, choose any $g\in H^a_K(\TT^d)$; then by definition there exists a Cauchy sequence $\{g_n\}_{n\in \NN}$ in $C^\infty(\TT^d)/\sim$ such that $\|g_n - g \|_{K,a} \to 0$ as \added[id=ale3,remark={nr6}]{$n\to\infty$}. As $\{ g_n \}_{n\geq 1}$ is convergent under $\|\cdot\|_{K,a}$ we have $\sup_{n\in\NN} \|g_n\|_{K,a}^2 < \infty$. 
Denote by
\begin{equation*}
\widetilde{F}_{\xi} := (F, f_{\xi})_{K,a}, \qquad F : \TT^d \rightarrow \RR.
\end{equation*}
We have for $\xi\in\ZZ^d\setminus\{0\}$ fixed,
\[ |\widetilde{g_\ell}(\xi) - \widetilde{g_m}(\xi)|^2 \leq \sum_{\xi\in\ZZ^d\setminus\{0\}} |\widetilde{g_\ell}(\xi) - \widetilde{g_m}(\xi)|^2 = \|g_\ell - g_m\|_{K,a}^2 \to 0 \]
\added[id=ale3,remark={nr7}]{as $\ell,\,m\to\infty$.}
So in fact for $\xi$ fixed, $\{\widetilde{g_n}(\xi)\}_{n\geq 1}$ is a Cauchy sequence and thus has a limit $\widetilde{g}(\xi)$. We define $h := \sum_{\xi\in\ZZ^d\setminus\{0\}} \widetilde{g}(\xi) f_\xi$. Note we have for all $\ell \in\NN$:
\[ 
\begin{split}
\sum_{\xi\in\ZZ^d\setminus\{0\}} |\widetilde{g}(\xi)|^2 \1_{\xi\in\ZZ^d_\ell} & = 
\added[id=wio,remark=]{ \lim_{n\rightarrow \infty}  \sum_{\xi \in \mathbb{Z}^d\setminus \{0\}} |\widetilde{g}_n(\xi)|^2 \1_{\xi\in\ZZ^d_\ell}} \\
& \added[id=wio,remark=]{\leq \lim_{n\rightarrow \infty} \| g_n \|^2_{K,a}} \leq \sup_{n\in\NN} \|g_n\|^2_{K,a} <\infty. 
\end{split}
\]
Therefore $\| h\|_{K,a}^2 <\infty$, so $h\in H^a_K(\TT^d)$. Moreover, we have $g_m \to h$ \added[id=ale3,remark={nr8}]{as $m\to\infty$} in $H^a_K(\TT^d)$. This can be seen by applying Fatou's lemma:
\[ \|g_m - h\|_{K,a}^2 = \sum_{\xi\in\ZZ^d\setminus\{0\}} |\widetilde{g_m}(\xi) - \widetilde{g}(\xi)|^2 \leq \liminf_{\ell\to\infty} \sum_{\xi\in\ZZ^d\setminus\{0\}} |\widetilde{g_m}(\xi) - \widetilde{g_\ell}(\xi)|^2 \to 0. \]
In this way we see that we must have $g = h = \sum_{\xi\in\ZZ^d\setminus\{0\}} \widetilde{g}(\xi) f_\xi$, so $g$ has a $f_\xi$-Fourier expansion.
\end{proof}

Next we will define a Hilbert--Schmidt operator on $H^a_K(\TT^d)$. 
\begin{lemma}\label{lem:AWS}
Let $\eps> {d}/{4}-a$ and $a<0$. The norm defined by
\begin{equation}
\|\cdot\|_B:=\| (-\Delta)^{-(\eps + a)} \cdot \|_{K,a}
\end{equation}
is a Hilbert--Schmidt norm on $H_K^a(\TT^d)$.
\end{lemma}
\begin{proof}
First of all, note that for $a > b$ we have $\|\cdot\|_{K,b} \leq \|\cdot\|_{K,a}$, so $H^a_K(\TT^d)\subset H^b_K(\TT^d)$. Now recall that $T$ is a Hilbert--Schmidt operator on a Hilbert space $H$ if, for $\{f_i\}_{i\neq0}$ an orthonormal basis of $H$, we have
\[ \sum_{i=1}^\infty \| T f_i \|_H^2 <\infty. \]
Set $T := (-\Delta)^{b-a}$. We have the following for all $\nu\in\ZZ^d\setminus\{0\}$:
\[ \| (-\Delta)^{b-a} f_\nu\|_{K,a}^2 = \sum_{\xi\in\ZZ^d\setminus\{0\}} \widehat{K}(\xi) \|\xi\|^{4a} \left| \widehat{K}(\xi)^{-1/2} \|\xi\|^{2(b-a)} \|\xi\|^{-2a} \1_{\nu=\xi} \right|^2 = \|\nu\|^{4(b-a)}.  \]
In this way, we see that
\[ \sum_{\nu\in\ZZ^d\setminus\{0\}} \| T f_\nu \|_{K,a}^2 = \sum_{\nu\in\ZZ^d\setminus\{0\}} \|\nu\|^{4(b-a)} < \infty, \]
if and only if $4(b-a) < -d$ which is equivalent to $b<- {d}/{4} +a$.  We write $-\eps := b < 0$ with $\eps >  {d}/{4}-a$. Note that the Banach space completion of $H^{a}_K$ with respect to the measurable norm $\| (-\Delta)^{-(\eps+a)} \cdot \|_{K,a}$ is exactly $H^{-\eps}_K(\TT^d)$. Indeed, we have
\begin{align*}
 \| g \|^2_B = \| (-\Delta)^{-(\eps+a)} g \|^2_{K,\,a} &= \sum_{\xi\in\ZZ^d\setminus\{0\}} \widehat{K}(\xi) \|\xi\|^{4a} \left | \|\xi\|^{-2(\eps +a)} \widehat{g}(\xi) \right |^2  \\
 &= \sum_{\xi\in\ZZ^d\setminus\{0\}} \widehat{K}(\xi) \|\xi\|^{4a} \|\xi\|^{-4(\eps+a)} |\widehat{g}(\xi)|^2 \\
 &= \sum_{\xi\in\ZZ^d\setminus\{0\}} \widehat{K}(\xi) \|\xi\|^{-4\eps} |\widehat{g}(\xi)|^2.\qedhere 
 \end{align*}
\end{proof}
\begin{definition}[Definition of the limit field]Our AWS is the triple $(H^{a}_K,H^{-\eps}_K,\mu_{-\eps})$ where $\eps$ is as in Lemma~\ref{lem:AWS}. We will choose from now on $a:=-1$ and denote $\|\cdot\|_{K,\,-1}$ simply as $\|\cdot\|_K$. The measure $\mu_{-\epsilon}$ is the unique Gaussian law on $H_K^{-\eps}$ whose characteristic functional is given in~\eqref{eq:def_Xi}. The field associated to $\Phi$ will be called $\Xi^K$. 
\end{definition}

\subsection{Covariance kernels}
We are going to show that positive, real Fourier coefficients on $\ZZ_n^d$ correspond to a positive definite function $(x,y)\mapsto K_n(x-y)$ on $\ZZ_n^d\times\ZZ_n^d$ by proving the analog of Bochner's theorem on the discrete torus.
\begin{lemma}\label{lem:posdef}
The function $(x,y)\mapsto K_n(x,y)=K_n(x-y)$ on $(\ZZ_n^d)^2$ is \added[id=ale3,remark={nr9}]{symmetric and} positive definite, and thus a well-defined covariance function, if and only if the Fourier coefficients $\widehat{K_n}$ are real-valued, symmetric and positive.
\end{lemma}
\begin{proof}
Assume first that $K_n$ is symmetric and positive definite. Then we have for any function \added[id=ale3,remark={nr4. We must restrict to real-valued $c$'s}]{$c:\ZZ_n^d\to\RR$} that is not the zero function that
\[ 0 < \sum_{x,y\in\ZZ_n^d} K_n(x-y) c(x) \overline{c(y)}.\]
We then find
\begin{eqnarray*}
\sum_{x,y\in\ZZ_n^d} K_n(x-y) c(x){c(y)} &=& \sum_{x\in\ZZ_n^d} c(x) \sum_{y\in\ZZ_n^d} K_n(x-y) {c(y)} \\
&=& n^d \sum_{x\in\ZZ_n^d} c(x) \sum_{\xi\in\ZZ_n^d} \exp\left (2\pi\i x \cdot \frac{\xi}{n}\right)\widehat{K_n}(\xi) {\widehat{c}(\xi)} \\
& = & n^{2d} \sum_{\xi \in \ZZ^d_n} \widehat{K}_n(\xi) \widehat{c}(\xi) \left (\frac{1}{n^d}\sum_{x\in \ZZ^d_n}  c(x) \exp\left (2\pi\i x \cdot \frac{\xi}{n}\right) \right )\\
&=& n^{2d} \sum_{\xi\in\ZZ_n^d} \widehat{K_n}(\xi) \widehat{c}(\xi)^2 > 0.
\end{eqnarray*}

As this needs to hold for all functions $c:\ZZ_n^d\to\RR$, we necessarily have $\widehat{K_n}(\xi) \in\RR_{>0}$. Since $K_n$ is symmetric, we also have 
\[ \widehat{K_n}(\xi) = \frac{1}{n^d} \sum_{z\in\ZZ_n^d} K_n(z) \exp\left(-2\pi\i z \cdot \frac{\xi}{n}\right) = \frac{1}{n^d} \sum_{z\in\ZZ_n^d} K_n(-z) \exp\left(-2\pi\i (-z) \cdot \frac{\xi}{n}\right) = \widehat{K_n}(-\xi), \]
thus $\widehat{K_n}$ is symmetric on $\ZZ_n^d$. The other direction of the proof can be obtained in a similar way hence we will omit the proof here.
\end{proof} 
\subsection{Divisible sandpile model and odometer function}
A divisible sandpile configuration $s=(s(x))_{x\in \mathbb{Z}^d_n}$ is a map $s:\mathbb{Z}^d_n \rightarrow \mathbb{R}$ where $s(x)$ can be interpreted as the \textit{mass} or \textit{hole} at vertex $x\in \mathbb{Z}^d_n$. 
It is known \citep[Lemma~7.1]{LMPU} that for all initial configurations $s$ such that $\sum_{x\in \mathbb{Z}^d_n} s(x)=n^d$ the model will stabilize to the all 1 configuration.

We consider in this paper initial divisible sandpile configurations of the form~\eqref{eq:def_s}
where $(\sigma(x))_{x\in \mathbb{Z}^d_n}$ are multivariate Gaussians with mean 0 and stationary covariance matrix ${K}_n$ given by
\begin{equation}\label{eq:cov_sigma}
\mathbb{E}(\sigma(x)\sigma(y)) = {K}_n(x,\,y)
\end{equation}
where $K_n$ is as in Theorem~\ref{main}.
Let us remark that putting \added[id=ale3,remark=nr10]{${K_n}(z):=\1_{z=0}$} retrieves the case when the $\sigma$'s are i.i.d. 

We will study the following quantity. Let $u_n=(u_n(x))_{x\in \mathbb{Z}^d_n}$ denote the \textit{odometer} \citep[Section~1]{LMPU} corresponding to the divisible sandpile model specified by the initial configuration \eqref{eq:def_s} on $\mathbb{Z}^d_n$. In words, $u_n(x)$ denotes the amount of mass exiting from $x$ during stabilization. Let 
\[
g(x,y)=\frac{1}{n^d}\sum_{z\in \mathbb{Z}^d_n} g^z(x,y)
\]
with $g^z(x,y)$ the expected amount of visits of a simple random walk on $\ZZ^d_n$ starting from $x$ and visiting $y$ before being killed at $z$. The following characterization of $u_n$ is similar to ~\citet[Proposition~1.3]{LMPU} and follows a close proof strategy.
\begin{lemma}\label{lem:eta_field}
Let $(\sigma(x))_{x\in\ZZ_n^d}$ be a collection of centered Gaussian random variables with covariance given in~\eqref{eq:cov_sigma} and consider the divisible sandpile $s$ on $\ZZ_n^d$ given by \eqref{eq:def_s}.
Then the sandpile stabilizes to the all 1 configuration and the distribution of the odometer $u_n(x)$ is given by
\[ u_n \stackrel{d}{=} \left(\eta - \min_{z\in\ZZ_n^d} \eta(z)\right). \]
Here $(\eta(x))_{x\in\ZZ_n^d}$ is a collection of centered Gaussian random variables with covariance
\[ \EE[\eta(x)\eta(y)] = \frac{1}{(2d)^2} \sum_{z,z'\in\ZZ_n^d} K_n(z-z') g(z,x) g(z',y). \]
\end{lemma}

\begin{proof} 
By Lemma 7.1 in \cite{LMPU} the sandpile stabilizes \added[id=ale3,remark={nr11}]{to the all $1$ configuration} and the odometer $u_n$ satisfies
\[\begin{cases} \Delta u_n(z) = 1-s(z)&z\in\ZZ_n^d\\
 \min_{z\in\ZZ^d_n} u_n(z) = 0 . 
  \end{cases}
 \]
Setting
\added[id=ale3,remark={nr12}]{$$
v^z(y):=\frac{1}{2d}\sum_{z\in\ZZ_n^d} g^z(x,\,y)(s(x)-1)
$$}
and
\added[id=ale3,remark={nr12}]{$v(y) = {n^{-d}} \sum_{z\in\ZZ_n^d} v^z(y) = (2d)^{-1}\sum_{x\in\ZZ_n^d} g(x,y) (s(x)-1)$} we can see as in \citet[Proposition~1.3]{LMPU} that \added[id=ale3,remark={nr13}]{$u_n-v$} is constant. So \added[id=ale3,remark={nr13}]{$ u_n\overset{d}= v + c$} for some constant $c\in\RR$.\deleted[id=ale,remark={nr14}]{, but since {$\min u_n = 0$}, it must hold that $c=0$ a.s.}
Now, since each $v(x)$ is a linear combination of Gaussian random variables, $v$ is again Gaussian with covariance
\begin{equation} \label{cov}
 \EE[v(x)v(y)] = \frac{1}{(2d)^2} \sum_{z,z'\in\ZZ_n^d} g(z,x) g(z',y) \EE[(s(z)-1)(s(z')-1)].
 \end{equation}
Observe that 
\added[id=ale3,remark={nr15}]{$$ \sum_{w\in\ZZ_n^d}K_n(z-w) =  \sum_{w\in\ZZ_n^d} K_n(w) =:C'. $$}
The expectation in the summation~\eqref{cov} can be calculated in the following way:
\begin{align*}
\EE[(s(z)-1)(s(z')-1)]& =   K_n(z-z') - \frac{1}{n^d} \sum_{w\in\ZZ_n^d} [K_n(z'-w) + K_n(z-w)] + \frac{1}{n^{2d}} \sum_{w,w'\in\ZZ_n^d} K_n(w-w')\\
 &= K_n(z-z') - \frac{2C'}{n^d} + \frac{C'n^d }{n^{2d}}  = K_n(z-z') - \frac{C'}{n^d}.
\end{align*}
If we now plug this into \eqref{cov}, we obtain
\begin{eqnarray*}
 \EE[v(x)v(y)]  &=& \frac{1}{(2d)^2} \sum_{z,z'\in\ZZ_n^d} K_n(z-z')g(z,x) g(z',y) - \added[id=ale3,remark={nr16}]{\frac{C'}{n^d (2d)^2}} \left( \sum_{z\in\ZZ_n^d} g(z,x)\right)^2.
 \end{eqnarray*}
Call $$R := \frac{C'}{n^d(2d)^2}\left(\sum_{z\in\ZZ_n^d} g(z,x)\right)^2$$
and define $Y\sim\mathcal{N}(0,R)$ independent of $v$. Then
\begin{equation}
 \label{eq:useY}
 \left( v + Y\right)_{x\in\ZZ_n^d} \stackrel{d}{=} (\eta(x))_{x\in\ZZ_n^d} \end{equation}
where $(\eta(x))_{x\in\ZZ_n^d}$ is a collection of centered Gaussians with 
\[ \EE[\eta(x)\eta(y)] = \frac{1}{(2d)^2} \sum_{z,z'\in\ZZ_n^d} K_n(z-z')g(z,x) g(z',y). \]
Now since $u_n-v$ is constant and $\min u_n = 0$, we conclude from~\eqref{eq:useY} the desired statement:
\[ u_n \stackrel{d}{=} \left (\eta - \min_{z\in\ZZ_n^d}\eta(z) \right ).\qedhere\]
\end{proof}

The next Lemma is concerned with yet another decomposition of the odometer function, namely it allows us to express its covariance in terms of Fourier coordinates. It is the analog of~\citet[Proposition~4]{CHR17}.

\begin{lemma}\label{covv}
Let $u_n:\ZZ_n^d\to\RR_{\ge 0}$ be the odometer function as in Lemma~\ref{lem:eta_field}. Then
\[ u_n \stackrel{d}{=} \left( \chi - \min_{z\in\ZZ_n^d} \chi_z\right).\]
Here $(\chi_{z})_{z\in\ZZ_n^d}$ is a collection of centered Gaussians with covariance
\[ \EE[\chi_x\chi_y] = \sum_{\xi\in\ZZ_n^d\setminus\{0\}} \widehat{K_n}(\xi) \frac{\exp\left(2\pi\i (x-y)\cdot\frac{\xi}{n}\right)}{\lambda_\xi^2}. 
\]
\end{lemma}

\begin{proof}
We denote $g_x(\cdot) := g(\cdot,x)$. Subsequently we utilize Plancherel's theorem on the covariance matrix of $\eta$ (see Lemma~\ref{lem:eta_field}) to find
\[
\begin{split}
\EE[\eta(x)\eta(y)] &=\frac{1}{(2d)^2} \sum_{z,z'\in\ZZ_n^d} K_n(z-z') g(z,y) g(z',x) \\
& = \frac{n^d}{(2d)^2} \sum_{z\in\ZZ_n^d} g(z,y) \sum_{\xi\in\ZZ_n^d} \exp\left(2\pi\i z \cdot \frac{\xi}{n}\right)\widehat{K_n}(\xi) \overline{\widehat{g_x}(\xi)}\\
&= \frac{n^{2d}}{(2d)^2} \widehat{K_n}(0) \widehat{g_y}(0) \overline{\widehat{g_x}(0)} +  \frac{n^{2d}}{(2d)^2} \sum_{\xi\in\ZZ_n^d\setminus\{0\}} \widehat{K_n}(\xi) \widehat{g_y}(\xi) \overline{\widehat{g_x}(\xi)}\label{eq:some_line}.
\end{split}
\]
We find $\widehat{g_x}(0) = n^{-d}\sum_{z\in\ZZ_n^d} g(z,x)$, which does not depend on $x$. In this way, we see that the first term is constant, and thus \replaced[id=ale2]{it gives no contribution to the variance of $u_n$ due to the recentering by the minimum}{ vanishes} in a similar way to the proof of Proposition 1.3 of \cite{LMPU} and the proof of Proposition 4 in \cite{CHR17}. Considering now the second summand above, we recall Equation (20) in \cite{LMPU}, which states that for $\xi\neq0$,
\[ \widehat{g_x}(\xi) = -2dn^{-d} \lambda_\xi^{-1} \exp\left(-2\pi\i\xi\cdot \frac{x}{n}\right).\]
We obtain, \added[id=ale2,remark={we need to say here the first term is still there but since it is subtracted in the minimum it doesn't matter}]{up to a constant factor which we ignore due to the recentering,} 
\[ \EE[\chi_x\chi_y] = \sum_{\xi\in\ZZ_n^d\setminus\{0\}} \widehat{K_n}(\xi) \frac{\exp\left(2\pi\i (x-y)\cdot\frac{\xi}{n}\right)}{\lambda_\xi^2}. \]
To show the positive-definiteness of $\EE[\chi_x\chi_y]$ we will show that for any function $c:\ZZ_n^d \to \mathbb{R}$, such that $c$ is not the zero function, $\sum_{x,y\in\ZZ_n^d} \EE[\chi_x\chi_y] c(x) \overline{c(y)} > 0$. First of all, since $K_n(z-z')$ is positive definite, we conclude that $\widehat{K_n}$ is positive by Lemma~\ref{lem:posdef}. Next we find
\begin{eqnarray*}
\sum_{x,y\in\ZZ_n^d} \EE[\chi_x\chi_y] c(x) \overline{c(y)} &=& \sum_{x,y\in\ZZ_n^d} c(x) \overline{c(y)} \sum_{\xi\in\ZZ_n^d\setminus\{0\}} \widehat{K_n}(\xi) \frac{\ee^{2\pi\i (x-y) \cdot \frac{\xi}{n}}}{\lambda_\xi^2} \\
&=& n^{2d} \sum_{\xi\in\ZZ_n^d\setminus\{0\}} \frac{\widehat{K_n}(\xi)}{\lambda_\xi^2} |\widehat{c}(\xi)|^2 > 0.
\end{eqnarray*}
which concludes the proof.
\end{proof}

\section{Proof of Theorem~\ref{main}}\label{sec:proof_main}

In this section we will prove Theorem~\ref{main} using the fact that convergence in distribution for the fields $\Xi_n^K$ is \added[id=ale3,remark={nr19. I expanded this part as it was not clearly stated before and the referee was right in asking for more explanations}]{equivalent to showing~\cite[Section~2.1]{ledoux:talagrand}} 
\begin{itemize}
 \item tightness in $H_K^{-\eps}(\TT^d)$;
 \item characterising the limiting field.
\end{itemize}
\added[id=ale3]{While tightness is deferred to Subsection~\ref{subsec:tight}, we will now characterize the limiting distribution by proving}
that for all mean-zero $f \in C^\infty(\TT^d)/\sim$ we have, as $n$ goes to infinity, 
$$\langle a_n \Xi^K_n,f\rangle\stackrel{d}{\to}\langle\Xi^K,f\rangle\sim\mathcal{N}(0,\|f\|_K^2).$$
Observe first that
\begin{equation} 
\langle a_n \Xi_n^K,f\rangle = 4\pi^2(2d)^{-1}n^{-2} \sum_{z\in\TT_n^d} u_n(nz) \int_{B(z,\frac{1}{2n})} \! f(x) \ \mathrm{d} x 
\end{equation}
is a linear combination of Gaussians, and thus Gaussian itself for each $n$. In order to now prove the convergence $\langle a_n \Xi_n^K,f\rangle \stackrel{d}{\to} \langle\Xi^K,f\rangle$ it is enough to show convergence of the first and second moment of $\langle a_n \Xi_n^K,f\rangle$ for any mean-zero $f\in C^\infty(\TT^d) / \sim$. To this end, note that by Lemma~\ref{covv}
\[ u_n \stackrel{d}{=} (\chi_x + C')_{x\in\ZZ_n^d}, \]
so in fact 
\begin{eqnarray*}
\langle a_n \Xi_n^K,f \rangle &=& 4\pi^2 (2d)^{-1}n^{-2} \sum_{z\in\TT_n^d} u_n(nz) \int_{B(z,\frac{1}{2n})} \! f(x) \ \mathrm{d} x \\
&\stackrel{d}{=}& 4\pi^2(2d)^{-1} n^{-2} \sum_{z\in\TT_n^d} (\chi_{nz} + C') \int_{B(z,\frac{1}{2n})} \! f(x) \ \mathrm{d} x.
\end{eqnarray*}
Because $f$ is a mean-zero function, we can neglect the random constant $C'$ and with a slight abuse of notation we write
\[ \langle a_n \Xi_n^K,f\rangle = 4\pi^2 (2d)^{-1}n^{-2} \sum_{z\in\TT_n^d} \chi_{nz} \int_{B(z,\frac{1}{2n})} \! f(x) \ \mathrm{d} x, \]
with
\[ \EE[\chi_{nz}\chi_{nz'}] = \frac{1}{16 (2d)^2} \sum_{\xi\in\ZZ_n^d\setminus\{0\}} \widehat{K_n}(\xi) \frac{\exp(2\pi\i (z-z') \cdot \xi)}{\left( \sum_{i=1}^d \sin^2\left(\pi \frac{\xi_i}{n}\right)\right)^2}. \]
As we have $\EE[\chi_{nz}] = 0$ for all $z\in\TT_n^d$, it follows that $\EE[\langle a_n \Xi^K_n,f\rangle] = 0$ for all $n$. For the second moment note first that
\begin{eqnarray}
\langle a_n \Xi_n^K,f\rangle^2 &=& 16\pi^4 (2d)^{-2} n^{-4} \sum_{z,z'\in\TT_n^d} \chi_{nz}\chi_{nz'} \int_{B(z,\frac{1}{2n})} \! f(x) \ \mathrm{d} x \int_{B(z',\frac{1}{2n})} \! f(x') \ \mathrm{d} x' \nonumber\\
&=& 16\pi^4 (2d)^{-2}n^{-(2d+4)} \sum_{z,z'\in\TT_n^d} \chi_{nz}\chi_{nz'} f(z)f(z') + 16\pi^4 (2d)^{-2} n^{-(2d+4)} \sum_{z,z'\in\TT_n^d} \chi_{nz}\chi_{nz'} E_n(z) E_n(z') \nonumber\\
&+& 32 \pi^4 (2d)^{-2}n^{-(2d+4)} \sum_{z,z'\in\TT_n^d} \chi_{nz}\chi_{nz'} f(z) E_n(z')\nonumber\\
&=:&16\pi^4 (2d)^{-2}n^{-(2d+4)} \sum_{z,z'\in\TT_n^d} \chi_{nz}\chi_{nz'} f(z)f(z')+\nonumber\\
&&+R_n^2+8\pi^2 (2d)^{-1}n^{-\frac{d+4}{2}}\sum_{z\in\TT_n^d}f(z)\chi_{nz}R_n.
\label{eq:terms}
\end{eqnarray}
Here we have defined\added[id=ale3,remark={nr20.Changed later too}]{}
\begin{align*} 
E_n(z) &:= \left( \int_{B(z,\frac{1}{2n})} \! n^d f(x) \ \mathrm{d}x - f(z)\right),\\
R_n&:=4\pi^2(2d)^{-1}n^{-(d+2)}\sum_{z\in\TT_n^d}\chi_{nz} E_n(z).
\end{align*}
In $E_n$ we essentially approximate the integral with the value at the centre of the box $B(z,\frac{1}{2n})$. The proof will now proceed in two steps: we will first show that the first term of the right-hand side of~\eqref{eq:terms} goes to the desired limiting variance (Proposition \ref{mad1}). Then we will argue in~Proposition \ref{mad2} that the second term in~\eqref{eq:terms} goes to 0 in $L^2$ and, likewise, the third term after an application of the Cauchy--Schwarz inequality. 

\begin{proposition}\label{mad1}
We have 
\[ \lim_{n\to\infty} 16\pi^4(2d)^{-2} n^{-(2d+4)} \sum_{z,z'\in\TT_n^d} f(z)f(z') \EE[\chi_{nz}\chi_{nz'}] = \|f\|_{K,}^2, \]
where $\|\cdot\|_{K}$ is defined in~\eqref{mad_a}.
\end{proposition}
\deleted[id=ale,remark={already explained what Prop. 9 is for}]{The next Proposition tells us that the approximation of the integral is indeed justified:}
\begin{proposition}\label{mad2}
We have that $\lim_{n\to\infty}R_n=0$ in $L^2(\TT^d)$.
\end{proposition}

First we will give the proof to Proposition~\ref{mad1}. As a preliminary tool, we need to recall the following bound on the eigenvalues $\lambda_\xi$.
\begin{lemma}[{\citet[Lemma 7]{CHR17}}]\label{lem:bound_eig}
There exists $c>0$ such that for all $n\in \mathbb{N}$ and $w\in \mathbb{Z}^d_n\setminus \{0\}$ we have
\[
\frac{1}{\| \pi w\|^4} \leq n^{-4} \left (\sum_{i=1}^d \sin^2 \left(\frac{\pi w_i}{n} \right) \right )^{-2} \leq \left( \frac{1}{\| \pi w \|^2} + \frac{c}{n^2}\right)^2.
\]
\end{lemma}
\begin{proof}[Proof of~Proposition~\ref{mad1}]
First of all,

If we now use the notation $f_n :\ZZ_n^d\to\RR$ for $f_n(\cdot) = f\left({\cdot}/{n}\right)$ then we find that
\[ n^{-d} \sum_{z\in\TT_n^d} f(z) \exp(\added[id=wio,remark=]{-} 2\pi\i z\cdot \xi) = \widehat{f_n}(\xi). \]
We have that 
\begin{equation}\label{eqn1}
16\pi^4 (2d)^{-2} n^{-(2d+4)} \sum_{z,z'\in \mathbb{T}^d_n} f(z)f(z')\mathbb{E}(\chi_{nz}\chi_{nz'}) = \pi^4 n^{-4} \sum_{\xi\in\ZZ_n^d\setminus\{0\}}  \frac{\widehat{K}_n(\xi)}{\left( \sum_{i=1}^d \sin^2\left(\pi \frac{\xi_i}{n}\right)\right)^2} |\widehat{f_n}(\xi)|^2.
 \end{equation}
For the next step in the proof we use Lemma~\ref{lem:bound_eig}. Since $\widehat{K}_n$ is non-negative, we have that on the one hand
\begin{equation}\label{eqn1-2}
\eqref{eqn1} \geq \sum_{\xi\in\ZZ_n^d\setminus\{0\}}  \frac{\widehat{K}_n(\xi)}{\| \xi \|^4} |\widehat{f_n}(\xi)|^2
\end{equation}
and on the other that
\begin{equation}\label{eqn3}
\begin{split}
\eqref{eqn1} & \leq \pi^4 \sum_{\xi\in\ZZ_n^d\setminus\{0\}}  \widehat{K}_n(\xi) |\widehat{f_n}(\xi)|^2 \left( \frac{1}{\| \pi \xi \|^2} + \frac{c}{n^2}\right)^2\\
& =: A+ B + C, 
\end{split}
\end{equation}
where
\begin{equation}
\begin{split}
A & := \sum_{\xi\in\ZZ_n^d\setminus\{0\}}  \frac{\widehat{K}_n(\xi)}{\| \xi \|^4} |\widehat{f_n}(\xi)|^2, \\
B & :=  C n^{-2} \sum_{\xi\in\ZZ_n^d\setminus\{0\}}  \frac{\widehat{K}_n(\xi)}{\| \xi \|^2} |\widehat{f_n}(\xi)|^2 ,\\
C& := C n^{-4} \sum_{\xi\in\ZZ_n^d\setminus\{0\}}  \widehat{K}_n(\xi) |\widehat{f_n}(\xi)|^2.\label{eq:three_terms}
\end{split}
\end{equation}
We will show in the following that $A$, which is the right-hand side of~\eqref{eqn1-2}, exhibits the following limit:
\begin{equation}\label{eq_conv}
\lim_{n \rightarrow \infty} \sum_{\xi\in\ZZ_n^d\setminus\{0\}}  \frac{\widehat{K}_n(\xi)}{\| \xi \|^4} |\widehat{f_n}(\xi)|^2 = \|f\|^2_{K}
\end{equation}
and $B,C$ vanish as $n\rightarrow \infty$.
As in \cite{CHR17} we split the proof into two cases. First we give a more direct proof for $d\leq 3$ and then consider $d \geq 4$. We have that $|\widehat{f_n}(\xi)|^2$ is uniformly bounded in $\xi$, so
\[ |\widehat{f_n}(\xi)| \leq n^{-d} \sum_{z\in\TT_n^d} |f(z)| \to \int_{\TT^d} \! |f(x)| \ \mathrm{d}x = \|f\|_{L^1(\TT^d)}<\infty.\]
We can now use the dominated convergence theorem  to obtain that
\[ \lim_{n\to\infty} \sum_{\xi\in\ZZ^d\setminus\{0\}} \1_{\xi\in\ZZ_n^d} \frac{\widehat{K}_n(\xi)}{\|\xi\|^4} |\widehat{f_n}(\xi)|^2  = \sum_{\xi\in\ZZ^d\setminus\{0\}} \frac{\widehat{K}(\xi)}{\|\xi\|^4} |\widehat{f}(\xi)|^2 = \|f\|_{K}^2. \]
In $d\geq 4$ we use a mollifying procedure. Take any $\phi\in\mathcal{S}(\RR^d)$, the space of Schwartz functions, such that $\phi$ is compactly supported on the unit cube $[-{1}/{2},\,{1}/{2})^d$ with integral 1. We write $\phi_\kappa(\cdot):= \kappa^{-d} \phi\left({\cdot}/{\kappa}\right)$ for $\kappa >0$. In order to show the convergence of \eqref{eq_conv}, we split the terms:
\begin{eqnarray*}
\sum_{\xi\in\ZZ^d\setminus\{0\}}\1_{\xi\in\ZZ_n^d} \frac{\widehat{K}_n(\xi)}{\|\xi\|^4} |\widehat{f_n}(\xi)|^2 = \sum_{\xi\in\ZZ^d\setminus\{0\}}\widehat{\phi_\kappa}(\xi) \1_{\xi\in\ZZ_n^d} \frac{\widehat{K}_n(\xi)}{\|\xi\|^4} |\widehat{f_n}(\xi)|^2 \\+ \sum_{\xi\in\ZZ^d\setminus\{0\}}\left(1-\widehat{\phi_\kappa}(\xi)\right) \1_{\xi\in\ZZ_n^d} \frac{\widehat{K}_n(\xi)}{\|\xi\|^4} |\widehat{f_n}(\xi)|^2.
\end{eqnarray*}
Next, we take from \cite{CHR17} the following bound, where $C>0$ is some constant:
\[ \left| \widehat{\phi_\kappa}(\xi) - 1\right|\leq C\kappa \|\xi\|. \]
Plugging this into the above, we find
\begin{eqnarray}
 \left| \sum_{\xi\in\ZZ^d\setminus\{0\}}\left(1-\widehat{\phi_\kappa}(\xi)\right) \1_{\xi\in\ZZ_n^d} \frac{\widehat{K}_n(\xi)}{\|\xi\|^4} |\widehat{f_n}(\xi)|^2\right| \leq \sum_{\xi\in\ZZ_n^d\setminus\{0\}} \left| \widehat{\phi_\kappa}(\xi) - 1\right| \frac{\widehat{K}_n(\xi)}{\|\xi\|^4} |\widehat{f_n}(\xi)|^2\nonumber\\
 \leq C\kappa\sum_{\xi\in\ZZ_n^d\setminus\{0\}} \frac{\widehat{K}_n(\xi)}{\|\xi\|^3} |\widehat{f_n}(\xi)|^2 \leq C\kappa \sum_{\xi\in\ZZ_n^d\setminus\{0\}} |\widehat{f_n}(\xi)|^2.\label{eq:cerchio_magico}
 \end{eqnarray}
 Note that by Plancherel theorem
 \begin{align*}
  \sum_{\xi\in\ZZ_n^d\setminus\{0\}}& |\widehat{f_n}(\xi)|^2\le n^{-d}\sum_{\xi\in\ZZ_n^d}\left|f\left(\frac{\xi}{n}\right)\right|^2=n^{-d}\sum_{\xi\in\TT_n^d}\left|f\left({\xi}\right)\right|^2.
 \end{align*}
Since the right-hand side above converges to $\int_{\TT^d}|f(x)|^2\De x<\infty$, one has
\begin{equation}\label{eq:bound_sum_f}\sum_{\xi\in\ZZ_n^d\setminus\{0\}} |\widehat{f_n}(\xi)|^2<\infty
\end{equation}
uniformly in $n$. Now taking the limit $\kappa\to0$ in~\eqref{eq:cerchio_magico} gives that
 \[ \lim_{\kappa\to0} \limsup_{n\to\infty} \sum_{\xi\in\ZZ^d\setminus\{0\}}\left(1-\widehat{\phi_\kappa}(\xi)\right) \1_{\xi\in\ZZ_n^d} \frac{\widehat{K}_n(\xi)}{\|\xi\|^4} |\widehat{f_n}(\xi)|^2 = 0.\]
 For the other term, we observe that since $\phi_\kappa$ is smooth, $\widehat{\phi_\kappa}(\xi)$ decays rapidly in $\xi$, and since the $\widehat{f_n}(\xi)$ are uniformly bounded, we can use the dominated convergence theorem to obtain
 \[ \lim_{\kappa\to0}\lim_{n\to\infty} \sum_{\xi\in\ZZ^d\setminus\{0\}}\widehat{\phi_\kappa}(\xi) \1_{\xi\in\ZZ_n^d} \frac{\widehat{K}_n(\xi)}{\|\xi\|^4} |\widehat{f_n}(\xi)|^2 = \lim_{\kappa\to0} \sum_{\xi\in\ZZ^d\setminus\{0\}} \widehat{\phi_\kappa}(\xi) \frac{\widehat{K}(\xi)}{\|\xi\|^4} |\widehat{f}(\xi)|^2.\]
We have that $|\widehat{\phi_\kappa}(\xi)|\leq 1$ and $\widehat{\phi_\kappa}(\xi)\to 1$ as $\kappa\to0$. Applying the dominated convergence theorem once more, we see that
\[ \lim_{\kappa\to0} \sum_{\xi\in\ZZ^d\setminus\{0\}} \widehat{\phi_\kappa}(\xi) \frac{\widehat{K}(\xi)}{\|\xi\|^4} |\widehat{f}(\xi)|^2 = \sum_{\xi\in\ZZ^d\setminus\{0\}} \frac{\widehat{K}(\xi)}{\|\xi\|^4} |\widehat{f}(\xi)|^2. \]
To conclude it remains to show that $B,\,C $ as defined in~\eqref{eq:three_terms} vanish. This can be achieved in a similar way to~\citet[proof of Proposition~5]{CHR17}. The key point is to observe that
\[
  \sum_{\xi\in\ZZ_n^d\setminus\{0\}} {\widehat{K_n}(\xi)} |\widehat{f_n}(\xi)|^2 \leq C
\]
is uniformly bounded in $n$, thanks to the uniform upper bound on $\widehat{K_n}(\cdot)$ and~\eqref{eq:bound_sum_f}.
\end{proof}

\begin{remark}
One could construct a different approximation $f_n(\cdot)$ of $f(\cdot/n)$ by considering the Taylor expansion of the latter. Being the supremal error between $f_n$ and $f(\cdot/n)$ of order $n^{-2d}$, one could use the fast decay of the Fourier coefficients of $f_n$ to apply dominated convergence directly in~\eqref{eq_conv}. This would give an alternative proof to the mollifying procedure, valid in all dimensions.
\end{remark}
We continue and prove Proposition~\ref{mad2}.
\begin{proof}[Proof of Proposition~\ref{mad2}]
Let us calculate $\EE[R^2_n]$ in the following way:
\begin{eqnarray*}
\EE[R_n^2] &=& 16\pi^4 (2d)^{-2}n^{-(2d+4)} \sum_{z,z'\in\TT_n^d} \EE[\chi_{nz}\chi_{nz'}] E_n(z) E_n(z')\\
&\leq& n^{-2d} \sum_{z,z'\in\TT_n^d} \sum_{\xi\in\ZZ_n^d\setminus\{0\}} \widehat{K}_n(\xi) \frac{\exp(2\pi\i (z-z')\cdot \xi)}{\|\xi\|^4} E_n(z) E_n(z') \\
&\leq& C n^{-2d} \sum_{\xi\in\ZZ_n^d\setminus\{0\}}\sum_{z,z'\in\TT_n^d} E_n(z) E_n(z')\exp(2\pi\i (z-z')\cdot \xi),
\end{eqnarray*}
because the $\widehat{K}_n(\xi)$ are uniformly bounded and $\|\xi\|\geq 1$. Now write $E'_n(x):=E_n\left({x}/{n}\right)$. Then the term above becomes
\begin{align*}
\sum_{\xi\in\ZZ_n^d\setminus\{0\}} \widehat{E'_n}(\xi) \overline{\widehat{E'_n}(\xi)} &\leq \sum_{\xi\in\ZZ_n^d} \widehat{E'_n}(\xi) \overline{\widehat{E'_n}(\xi)} \\
&= n^{-d} \sum_{\xi\in\ZZ_n^d} E'_n(\xi) \overline{E'_n(\xi)} \\
& \leq \|E_n\|_{L^\infty(\TT^d)}^2 \leq Cn^{-2}\to 0.
\end{align*}
The last inequality relies on the bound given in \citet[Lemma 8]{CHR17}:
\[ \sup_{z\in\TT_n^d} |E_n(z)|\leq Cn^{-1}. \]
This then concludes the proof to Proposition~\ref{mad2}.
\end{proof}

\subsection{Tightness in \texorpdfstring{$H_K^{-\eps}$}{}}\label{subsec:tight}
To complete the proof, we show that the convergence in law $a_n\Xi^K_n \stackrel{d}{\to} \Xi^K$ as $n\to \infty$ holds in the Sobolev space $H_K^{-\eps}(\TT^d)$ for any $\eps>\max\{1+{d}/{4},\,{d}/{2}\}$. We state the following Theorem:
\begin{theorem}\label{tight}
The sequence $(a_n\Xi^K_n)_{n\in\NN}$ is tight in $H_K^{-\eps}(\TT^d)$, \replaced[id=ale3,remark={nr24}]{in fact}{in other words}, for all $\delta>0$ there exists $R_\delta>0$ such that
\[ \sup_{n\in\NN} \PP\left( \left\|a_n\Xi^K_n\right\|_{H_K^{-\frac{\eps}{2}}} \geq R_\delta\right) \leq \delta.\]
\end{theorem}
\begin{proof} The proof of this Theorem is analogous to the proof of tightness in \citet[Section~4.2]{CHR17}. We first apply Markov's inequality and see
\[ \PP\left(\left\|a_n\Xi^K_n\right\|_{H_K^{-\frac{\eps}{2}}} \geq R_\delta\right) \leq \frac{\EE\left[\left\|a_n\Xi^K_n\right\|_{H_K^{-\frac{\eps}{2}}}^2\right]}{R_\delta^2}.\]
Now whenever we have 
\[ \sup_{n\in\NN} \EE\left[\left\| a_n\Xi^K_n\right\|_{H_K^{-\frac{\eps}{2}}}^2\right] \leq C, \]
the assertion follows as we can choose $R_\delta$ such that
\[ \PP\left( \left\|a_n\Xi^K_n\right\|_{H_K^{-\frac{\eps}{2}}} \geq R_\delta\right) \leq \frac{C}{R_\delta^2}<\delta.\]
We calculate the expectation and obtain
\begin{eqnarray*}
\EE\left[\left\|a_n\Xi^K_n\right\|^2_{H_K^{-\frac{\eps}{2}}}\right] &=& a_n^2 \sum_{\added[id=wio,remark=]{\xi} \in\ZZ_n^d\setminus\{0\}} \sum_{x,y\in\TT_n^d} \frac{\widehat{K}_n(\xi)}{\|\added[id=wio,remark=]{\xi}\|^{2\eps}}  \EE[\chi_{nx}\chi_{ny}] \int_{B(x,\frac{1}{2n})} \!\phi_{\added[id=wio,remark=]{\xi}} (\vartheta) \ \mathrm{d}\vartheta \int_{B(y,\frac{1}{2n})} \! \overline{\phi_{\added[id=wio,remark=]{\xi}} (\vartheta)} \ \mathrm{d}\vartheta 
\end{eqnarray*}
Now define 
\begin{align*}F_{n,\added[id=wio,remark=]{\xi}}:\TT_n^d&\to\RR\\x & \mapsto F_{n, \added[id=wio,remark=]{\xi}}(x):= \int_{B(x,\frac{1}{2n})}\! \phi_{\added[id=wio,remark=]{\xi}}(\vartheta) \ \mathrm{d}\vartheta.\end{align*}
We have that both $\1_{B(x,\frac{1}{2n})}$ and $ \phi_{\added[id=wio,remark=]{\xi}} \in L^2(\TT^d)$ so by Cauchy-Schwarz inequality $F_{n, \added[id=wio,remark=]{\xi}}\in L^1(\TT^d)$. Next we claim that, for some $C'>0$,
\begin{equation}\label{claim}
 \sup_{\added[id=wio,remark=]{\xi} \in\ZZ^d}\sup_{n\in\NN} \left|\sum_{x,y\in\TT_n^d} n^{-4} \EE[\chi_{nx}\chi_{ny}]  F_{n,\added[id=wio,remark=]{\xi}}(x) \overline{F_{n, \added[id=wio,remark=]{\xi}}(y)}\right| \leq C'.
 \end{equation}
Remark that similarly to \citet[Equation (4.5)]{CHR17}, we have \begin{equation}\label{eq:use_bound}
|n^{-4}\lambda_{\xi}^{-2}| \leq C\|\xi\|^{-4}
\end{equation}                                                                 
for some $C>0$. We write $G_{n, \added[id=wio,remark=]{\xi}}:\ZZ_n^d\to\RR$ for $G_{n, \added[id=wio,remark=]{\xi}}(\cdot):=F_{n, \added[id=wio,remark=]{\xi}}\left({\cdot}\,{n}\right)$. Using this, we find
\begin{align*}
&\left| \sum_{x,y\in\TT_n^d} n^{-4} \EE[\chi_{nx}\chi_{ny}]  F_{n, \added[id=wio,remark=]{\xi}}(x) \overline{F_{n, \added[id=wio,remark=]{\xi}}(y)} \right|\\
&= \left| \sum_{x,y\in\TT_n^d} n^{-4} \sum_{z \in\ZZ_n^d\setminus\{0\}} \widehat{K_n}(z) \frac{\exp(2\pi\i (x-y)\cdot z)}{\lambda_z^2} F_{n, \added[id=wio,remark=]{\xi}}(x) \overline{F_{n, \added[id=wio,remark=]{\xi}}(y)} \right| \\
&= \left| \sum_{z \in\ZZ_n^d\setminus\{0\}} n^{-4} \lambda_z^{-2} \widehat{K_n}(z) n^{2d} |\widehat{G_{n,\added[id=wio,remark=]{\xi}}}(z)|^2\right|\\
&\stackrel{\eqref{eq:use_bound}}{\le} C n^{2d} \sum_{z \in\ZZ_n^d\setminus\{0\}} \|z\|^{-4} |\widehat{G_{n, \added[id=wio,remark=]{\xi}}}(z)|^2.
\end{align*}
Here we have exploited the fact that $\sup_{z \in\ZZ_n^d}\widehat{K_n}(z)<\infty$ by~\eqref{eq:K_hat}. Now by the triangle inequality,
\[|F_{n, \added[id=wio,remark=]{\xi}}(w)|  = \left| \int_{B(w,\frac{1}{2n})} \! \phi_{\added[id=wio,remark=]{\xi}}(\vartheta) \ \mathrm{d}\vartheta \right| \leq \int_{B(w,\frac{1}{2n})}\! \ \mathrm{d} \vartheta = n^{-d}.\]
Thus
\begin{eqnarray*}
\sum_{z\in\ZZ_n^d\setminus\{0\}} \|z\|^{-4} |\widehat{G_{n, \added[id=wio,remark=]{\xi}}}(z)|^2 &\leq& \sum_{z\in\ZZ_n^d} |\widehat{G_{n,\added[id=wio,remark=]{\xi}}}(z)|^2 = n^{-d} \sum_{z\in\ZZ_n^d} G_{n, \added[id=wio,remark=]{\xi}}(z) \overline{G_{n, \added[id=wio,remark=]{\xi}}(z)} \\
&=& n^{-d} \sum_{z'\in\TT_n^d} F_{n, \added[id=wio,remark=]{\xi}}(z') \overline{F_{n, \added[id=wio,remark=]{\xi}}(z')} \leq n^{-2d} \sum_{z'\in\TT_n^d} 
\int_{B(z',\frac{1}{2n})}\! |\phi_{\added[id=wio,remark=]{\xi}}(\vartheta)| \ \mathrm{d}\vartheta \\
&\leq& n^{-2d} \|\phi_{\added[id=wio,remark=]{\xi}}\|_{L^1(\TT^d)} = Cn^{-2d}. 
\end{eqnarray*}
We then use this bound to obtain
\[ C n^{2d}  \sum_{z\in\ZZ_n^d\setminus\{0\}} \|z\|^{-4} |\widehat{G_{n,\added[id=wio,remark=]{\xi}}}(z)|^2 \leq Cn^{2d}  n^{-2d} = C. \]
This is a constant that does not depend on $n$ or $\xi$, so the claim (\ref{claim}) is proven. Using it, we have by the Euler-Maclaurin formula and the boundedness of $\widehat K(\cdot)$
\begin{eqnarray*}
\EE\left[\left\|a_n \Xi_n^K\right\|_{H_K^{-\frac{\eps}{2}}}^2\right] = \sum_{\added[id=wio,remark=]{\xi} \in\ZZ^d\setminus\{0\}} \frac{\widehat{K}(\xi)}{\|\added[id=wio,remark=]{\xi}\|^{2\eps}} \sum_{x,y\in\TT_n^d}  \EE[\chi_{nx}\chi_{ny}]  F_{n, \added[id=wio,remark=]{\xi}}(x) \overline{F_{n, \added[id=wio,remark=]{\xi}}(y)} \\
\leq  C'\sum_{k\geq1} k^{d-1-2\eps} \leq C.
\end{eqnarray*}
The last estimate is due to the fact that $-\eps < -{d}/{2}$.
\end{proof}

\section{Some examples}\label{sec:examples}

In this Section we want to give some concrete examples of initial distributions and Gaussian fields that can be generated via scaling limits of the odometer.

\subsection{Initial Gaussian distributions with power-law covariance}
We would like to look at the case in which the initial distribution of the $\sigma$'s is $(\sigma(x))_{x\in \mathbb{Z}^d_n} \sim \mathcal N(0,\, K_n)$ when
the covariance matrix $K_n$ is polynomially decaying. As an example consider
\[
K_n^{\pm}(x,y) = \begin{cases} 7 & \text{ if } x=y \\  \pm \| x-y\|^{-3} & \text{ otherwise }.\end{cases}
\]
We choose $K(0)=7$ in order to make the covariance matrix positive definite.
The corresponding realizations of the odometer function are indicated in Figures~\ref{Fig:PosK}-\ref{Fig:NegK} and superposed in Figure~\ref{fig_super}.
\begin{figure}[h]
\centering
 \centering
    \begin{minipage}{.5\textwidth}
\includegraphics[scale=1]{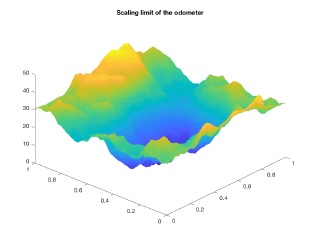}
\caption{Rescaled odometer associated to $K_n^+$.}\label{Fig:PosK}
\end{minipage}%
    \begin{minipage}{.5\textwidth}
\centering
\includegraphics[scale=1]{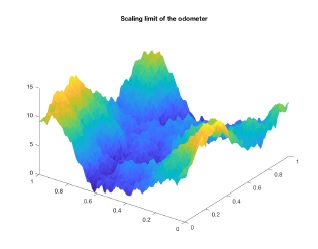}
\caption{Rescaled odometer associated to $K^-_n$.}\label{Fig:NegK}
\end{minipage}
\end{figure}
\begin{figure}[htb]
\includegraphics[scale=0.7]{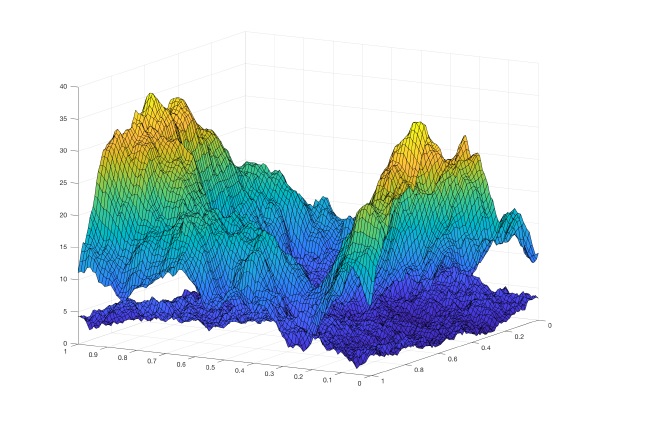}
\caption{Odometer interfaces associated to $K^+_n$ and $K^-_n$.}
\label{fig_super}
\end{figure}

\subsection{A bi-Laplacian field in the limit}The next Proposition shows that, even if~\eqref{eq:K_hat} does not hold, but instead $
 \lim_{n\to\infty} n^d \widehat{K_n}(\xi)$ exists and is finite
for all $\xi$, one can rescale the weights $\sigma$ to go back to the setting of Theorem~\ref{main}. 
\begin{proposition}
Consider the divisible sandpile configuration defined in~\eqref{eq:def_s} with $(\sigma(x))_{x\in \mathbb{Z}^d_n}$ a sequence of centered multivariate Gaussians with covariance $K_n$ satisfying
\begin{itemize}
 \item $\lim_{n\to\infty}K_n(\added[id=wio,remark=]{w})=K(\added[id=wio,remark=]{w})$ exists for all $\added[id=wio,remark=]{w} \in\ZZ^d,$
 \item $K\in\ell^1(\ZZ^d)$,
 \item  $C_K:=\sum_{z\in\ZZ^d}K(z)\neq 0$.
\end{itemize}
Let $(u_n(x))_{x\in\ZZ_n^d}$ be the associated odometer and furthermore $b_n:= 4\pi^2 (2d)^{-1}n^{(d-4)/2}C_K^{-1/2}$. We define the formal field on $\TT^d$ by
\[ \Xi^K_n(x):=  \sum_{z\in\TT_n^d} u_n(nz) \1_{B(z,\frac{1}{2n})}(x), \ \ \ x\in\TT^d. \]
Then $b_n \Xi^K_n $ converges in law as $n\to\infty$ to $\Xi^K=FGF_2(\mathbb{T}^d)$ in the topology of the Sobolev space $H_K^{-\eps}(\TT^d)$, where  $\eps>\max\left\{{d}/{2},\,{d}/{4}+1\right\}$.
\end{proposition}
\begin{proof}
The basic idea is to use, rather than the weights $\sigma$ as in the assumptions, the rescaled weights
\[
 \sigma^\prime(x):=n^{d/2}\sigma(x),\,\quad x\in\ZZ^d_n.
\]
The Fourier transform of the associated covariance kernel is now $\widehat{K^\prime_n}=n^{d}\widehat{K_n}$. Now observe that
\begin{align*}
 \widehat{K^\prime_n}(\xi)=\sum_{w\in \ZZ^d}K(w)\1_{w\in\ZZ^d_n}\ee^{\added[id=wio,remark=]{-}2\pi\i w\cdot\xi/n}.
\end{align*}
By dominated convergence and the fact that $K\in\ell^1(\ZZ^d)$ we deduce
\begin{equation}\label{eq:lim_K'_n}
\lim_{n\to\infty} \widehat{K^\prime_n}(\xi)=C_K,\quad\xi\in \ZZ^d.
\end{equation}
\replaced[id=ale,remark={here above we had $\xi\neq 0$ but it holds for $0$ too right?}]{}{a}
We repeat the computation of~\eqref{eq:terms} for $\langle  b_n\Xi^K_n,\,f\rangle^2$ and we obtain as leading term
\[
 16 (2d)^{-2} \pi^4C_K^{-1} n^{-(d+4)} \sum_{z,z'\in\TT_n^d} f(z)f(z') \chi_{nz}\chi_{nz'}.
\]
To compute the variance $\EE[\langle  b_n\Xi^K_n,\,f\rangle^2]$, we have
\begin{eqnarray}
16 (2d)^{-2} \pi^4 C_K^{-1}n^{-(d+4)}&& \sum_{z,z'\in\TT_n^d} f(z)f(z') \EE[\chi_{nz}\chi_{nz'}] \nonumber
\\&=& C_K^{-1}\pi^4 n^{-(d+4)} \sum_{z,z'\in\TT_n^d} f(z)f(z')\sum_{\xi\in\ZZ_n^d\setminus\{0\}} \widehat{K}_n(\xi) \frac{\exp(2\pi\i (z-z') \cdot \xi)}{\left( \sum_{i=1}^d \sin^2\left(\pi \frac{\xi_i}{n}\right)\right)^2} \nonumber\\
&=&C_K^{-1}\pi^4 n^{-(2d+4)}\sum_{\xi\in\ZZ_n^d\setminus\{0\}} \frac{\widehat{K^\prime_n}(\xi)}{\left( \sum_{i=1}^d \sin^2\left(\pi \frac{\xi_i}{n}\right)\right)^2} \sum_{z,z'\in\TT_n^d} f(z) f(z') \ee^{2\pi\i (z-z') \cdot \xi}\label{eq:one_step}.
\end{eqnarray}
From this point onwards the proof of Proposition~\ref{mad1} applies verbatim. In view of~\eqref{eq:lim_K'_n}, the rescaling $C_K^{-1}$ in the variance is done to obtain a limiting field with characteristic functional
\[
 \Phi(f):=\exp\left(-\frac12\sum_{\xi\in \ZZ^d\setminus\{0\}}\|\xi\|^{-4}|\widehat f(\xi)|^2\right)
\]
which identifies the $FGF_2(\TT^d)$.
\end{proof}

\subsection{Fractional limiting fields}
We can now use Theorem \ref{main} to construct $(1+s)$-Laplacian limiting fields for arbitrary $s\in (0,\infty)$. 
Define an initial sandpile configuration such that 
\begin{equation} \label{eq:bil_sim}
\lim_{n\to\infty}\widehat{K_n}(\xi)=\widehat{K}(\xi) := \|\xi\|^{-4s},\quad\xi\in\ZZ^d\setminus\{0\}.
\end{equation}

We do need $\widehat{K}(0)>0$ to ensure positive definiteness of the kernel, so one can choose any arbitrary constant to satisfy this constraint. Using the above Fourier multiplier (observe that it is indeed uniformly bounded for $s\in (0,\infty)$ and hence satisfies~\eqref{eq:K_hat}) and applying Theorem \ref{main}, we find the following limiting distribution: for all $f\in C^\infty(\TT^d)/\sim$ we have that $\langle\Xi^K,f\rangle$ is a centered Gaussian with variance
\begin{eqnarray*}
 \EE\left[\langle\Xi^K,f\rangle^2\right] = \sum_{\xi\in\ZZ^d\setminus\{0\}} \|\xi\|^{-4s} \|\xi\|^{-4} |\widehat{f}(\xi)|^2 &=& \sum_{\xi\in\ZZ^d\setminus\{0\}} \|\xi\|^{-4(s+1)} |\widehat{f}(\xi)|^2\\
 &=& (f,(-\Delta)^{-2(s+1)} f)_{L^2(\TT^d)}\\
& = &\| (-\Delta)^{-(s+1)} f\|^2_{L^2(\mathbb{T}^d)}.
 \end{eqnarray*}

\begin{figure}[h]
\centering
\includegraphics[width=6cm]{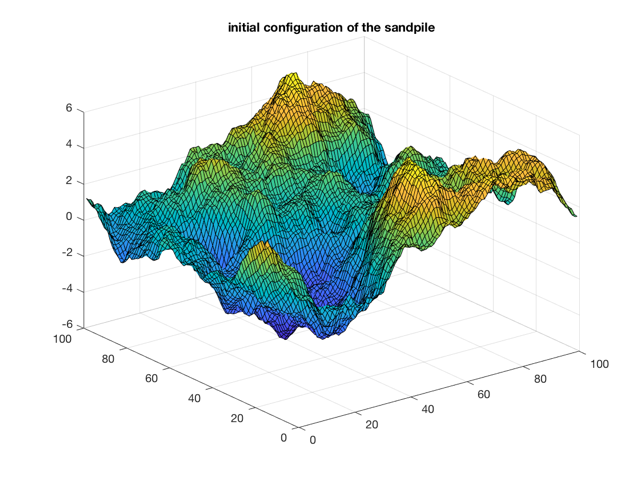}
\hspace{0.7cm}
\includegraphics[width=6cm]{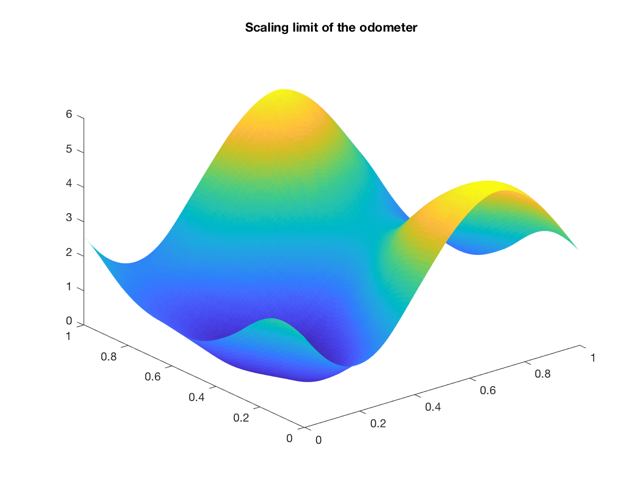}
\caption{The initial configuration $s$ obtained from~\eqref{eq:bil_sim} as in Theorem~\ref{main} and the limiting field of the associated odometer for $s:={1}/{2}$.}
\label{fig2}
\end{figure}


\subsection{Initial Gaussian distributions with fractional Laplacian covariance}

We have seen in the previous example that if we define an initial configuration on $\ZZ_n^d$ with covariance given by $K_n(\cdot)= \left(\|\cdot\|^{-4s}\right)^\vee$, where the inverse Fourier transform is on $\ZZ_n^d$, the limiting field is Gaussian with variance (for every $f\in C^\infty(\TT^d)$ with mean zero):
\[ \EE\left[\langle\Xi^K,f\rangle^2\right] = (f,(-\Delta)^{-2(s+1)} f)_{L^2(\TT^d)}. \]
However, the covariance $K_n$ does not necessarily agree with that of the discrete $s$-Laplacian field on $\ZZ_n^d$. Namely, we recall the definition of (minus) the discrete fractional Laplacian $(-\Delta_g)^{s}$ on $\ZZ_n^d$ \cite[Section~2.1]{LongRange}:
\begin{equation}\label{fraLap}
 -(-\Delta_g)^{\added[id=wio,remark=]{s/2}}f(x):=\sum_{y\in\ZZ_n^d}(f(x+y)+f(x-y)-2 f(x))p_n^{(s)}(y)
\end{equation}
where the weights $p_n^{(s)}(\cdot)$ are given by
\[ p_n^{(s)}(x,\,y)= p_n^{(s)}(0,\,x-y):=c^{(s)}\sum_{\genfrac{}{}{0pt}{}{z\in\ZZ^d\setminus\{0\}}{z\equiv x-y\text{ mod }\ZZ_n^d}}\|z\|^{-d-s}
\]
and $c^{(s)}$ is the normalizing constant. The above representation has the advantage that we have an immediate interpretation of the fractional graph Laplacian in terms of random walks.

We introduce the powers $(-\Delta_g)^s$ differently from~\eqref{fraLap}, in a way which is more convenient for us. Since our main working tools are Fourier analytical, we will define the discrete $s$-Laplacian $(-\Delta_g)^{-s}$ through its action in Fourier space.
Let $s>0$ and $f\in \ell^2(\ZZ_n^d) $ be such that
\begin{equation}
\label{eq:discrete_f}
 f(\cdot) =\sum_{\nu\in\ZZ^d_n\setminus\{0\}} \widehat{f}(\nu) \psi_\nu(\cdot) 
\end{equation}
where the functions $\psi_\nu$ were defined in~\eqref{def-fourier-basis-discrete} and $\widehat{f}(\nu):=\langle f,\,\psi_\nu\rangle$. We thus define the \emph{discrete fractional Laplacian} as
\begin{equation}\label{defFra2}
(-\Delta_g)^{-s} f(\cdot) := \sum_{\nu\in\ZZ_n^d\setminus\{0\}}(-\lambda_{\nu})^{-s} \widehat{f}(\nu)\psi_\nu(\cdot),\quad f\in\ell^2(\ZZ_n^d)
\end{equation}
having $\lambda_\nu$ as in~\eqref{eq:eigenvalues}. 
Note that for the above expression to be well-defined we need $f$ as in~\eqref{eq:discrete_f}, in other words that $\widehat{f}(0)=0$ which is equivalent to
\begin{displaymath}
\sum_{z\in \mathbb{Z}^d_n} f(z)=0.
 \end{displaymath}
 This space is the discrete analog of $C^\infty(\mathbb{T}^d) /\sim.$
The definition in \eqref{defFra2} resembles one of the possible ways to define the continuum fractional Laplacian \cite{Kwa10} and is akin to the definition of the zero-average discrete Gaussian free field \cite{Aba17}. Indeed, when $s=1$ the two definitions coincide.

\begin{proposition}\label{main4}
Let $s>0$ and let $(u_n(z))_{z\in\ZZ_n^d}$ denote the odometer function associated with the weights $(\sigma(z))_{z\in\ZZ_n^d}$, which are sampled from a jointly Gaussian distribution $\mathcal{N}\left(0,\,a_n^{2s} (-\Delta_g)^{-2s}\right)$ with $a_n:=4\pi^2 (2d)^{-1}n^{-2}$. Let us define the formal field $\Xi_n^K$ by
\[
 \Xi_n^K(x):=\sum_{z\in\TT_n^d} u_n(nz) \1_{B(z,\frac{1}{2n})}(x),\quad x\in\TT^d.
\]

Then as $n\to\infty$
the field $ a_n \Xi_n^K$ converges to $ \Xi^{K}$
in the Sobolev space $H^{-\eps}_{K}(\TT^d)$, $\eps > \max \{d/2,1+d/4\}$.  $\Xi^{K}$ is the Gaussian field on $\TT^d$ such  that for each $f\in C^\infty(\TT^d)/\sim$ we have $$\langle \Xi^{K},f\rangle\sim \mathcal{N}\left(0,\| (-\Delta)^{-(s+1)} f\|_{L^2(\TT^d)}^2\right).$$
\end{proposition}

\begin{proof}
In the notation of Theorem \ref{main},which we intend to apply here, we have 
\[ \widehat{K_n}(\xi) = \left( \frac{n^2}{4\pi^2} \right)^{-2s} \left(-\lambda_\xi\right)^{-2s},\quad \xi\in\ZZ_n^d\setminus\{0\}. \]
It is immediate that $\widehat{K_n}$ is even and positive. We show that $\widehat{K_n}(\xi)$ is bounded uniformly and converges to $\|\xi\|^{-4s}$. 
As the function $x\mapsto x^{s}$ is strictly increasing for $x\geq 0$, we can apply Lemma 7 from \cite{CHR17} and take $s$-powers such that the inequality still holds. This gives:
\[ \|\pi\xi\|^{-4s} \leq \left( n^{2} \sum_{i=1}^d \sin^2\left(\frac{\pi\xi_i}{n}\right)\right)^{-2s} \leq \left( \|\pi\xi\|^{-2} + cn^{-2} \right)^{2s}. \]
Observe first that since $n$ is fixed and $\|\xi\|\geq 1$, we indeed have that $a_n^{2s} (-\lambda_\xi)^{-2s}$ is uniformly bounded in both $n$ and $\xi$ (recall~\eqref{eq:eigenvalues}). Now taking the limit in the above, we have the convergence of $\widehat{K_n}(\xi)$ to $\|\xi\|^{-4s}.$
Therefore the assumptions of Theorem~\ref{main} are satisfied.
\end{proof}

\bibliographystyle{plainnat}
\bibliography{literaturDSM}

\end{document}